\documentclass{amsart}

\usepackage[nobysame]{amsrefs}
\usepackage{amsthm}
\usepackage{amssymb}

\usepackage{tikz-cd}
\usepackage{colonequals}
\usepackage{enumerate}
\usepackage{eucal}

\usepackage{hyperref}
\hypersetup{%
  bookmarksnumbered=true,%
  colorlinks=true,%
  linkcolor=blue,%
  citecolor=blue,%
  filecolor=blue,%
  menucolor=blue,%
  urlcolor=blue,%
  bookmarksopen=true,%
  bookmarksdepth=2,%
  pageanchor=true}

\makeatletter
\@namedef{subjclassname@2020}{\textup{2020} Mathematics Subject Classification}
\makeatother

\numberwithin{equation}{section}

\swapnumbers

\theoremstyle{plain}
\newtheorem{theorem}[equation]{Theorem}
\newtheorem{proposition}[equation]{Proposition}
\newtheorem{lemma}[equation]{Lemma} 
\newtheorem{corollary}[equation]{Corollary}

\theoremstyle{definition}

\newtheorem{chunk}[equation]{}

\theoremstyle{remark}

\newtheorem*{ack}{Acknowledgements}

\newcommand{\acat}{\mathsf{C}_{\mathcal{O}}}
\newcommand{\ann}{\operatorname{ann}}
\newcommand{\aqh}[4]{\operatorname{D}_{#1}({#2}/{#3};#4)}

\newcommand{\con}[1]{\Phi_{#1}}
\newcommand{\cmod}[1]{\Psi_{#1}}
\newcommand{\bcmod}[1]{\hat{\Psi}_{#1}}

\newcommand{\depth}{\operatorname{depth}}
\newcommand{\End}[2]{\operatorname{End}_{#1}(#2)}
\newcommand{\Ext}[4]{\operatorname{Ext}^{#1}_{#2}(#3,#4)}

\newcommand{\fitt}[2]{\operatorname{Fitt}_{#1}(#2)}
\newcommand{\gam}[1]{\varGamma_{#1}}

\newcommand{\Hom}[3]{\operatorname{Hom}_{#1}(#2,#3)}
\newcommand{\sHom}[3]{\underline{\operatorname{Hom}}_{#1}(#2,#3)}

\newcommand{\Ker}{\operatorname{Ker}}

\newcommand{\length}[2]{\operatorname{length}_{#1}{#2}}
\newcommand{\pos}[1]{[\![{#1}]\!]}
\newcommand{\rank}{\operatorname{rank}}

\newcommand{\vf}{\varphi}

\newcommand{\fa}{\mathfrak{a}} 
\newcommand{\fb}{\mathfrak{b}} 
\newcommand{\fm}{\mathfrak{m}} 
\newcommand{\fp}{\mathfrak{p}}
\newcommand{\fq}{\mathfrak{q}}

\newcommand{\lra}{\longrightarrow}
\newcommand{\xra}{\xrightarrow}

\newcommand{\bs}{\boldsymbol}
\newcommand{\mco}{\mathcal{O}}

\begin{document}

\title[Wiles defect]{Wiles defect for modules and \\ criteria for freeness}

\author[S.~Brochard]{Sylvain Brochard}
\address{IMAG, University of Montpellier, CNRS, Montpellier, France}
\email{sylvain.brochard@umontpellier.fr}

\author[S.~B.~Iyengar]{Srikanth B.~Iyengar}
\address{Department of Mathematics,
University of Utah, Salt Lake City, UT 84112, U.S.A.}
\email{iyengar@math.utah.edu}

\author[C.~B.~Khare]{Chandrashekhar  B. Khare}
\address{Department of Mathematics,
University of California, Los Angeles, CA 90095, U.S.A.}
\email{shekhar@math.ucla.edu}

\date{\today}


\begin{abstract} 
F.~Diamond proved a numerical criterion for modules over local rings to be free modules over complete intersection rings. 
We formulate  a refinement of these results using the notion of {\it Wiles defect}. A key step in the proof is a formula that expresses the Wiles defect of a module in terms of the Wiles defect of the underlying ring. 
\end{abstract}

\maketitle
   
\section{Introduction}
In his work~\cites{Wiles:1995} on  modularity of elliptic curves and Fermat's Last Theorem, Wiles discovered a numerical criterion for certain noetherian local rings $A$  to be complete intersections. Diamond~\cite{Diamond:1997} generalized Wiles' result by establishing a criterion for modules $M$  over  $A$ to be free and for $A$ to be a complete intersection; see the discussion  below for the precise statements of their results.   

To set the stage for our work we recall the  number theoretic application  of the numerical criterion,  although  this paper concerns only its commutative algebraic aspects.  The ring of interest  is a  deformation ring $R$   associated to a modular  representation  $\overline \rho\colon G_{\mathbb Q} \to GL_2(k)$, with $G_\mathbb Q$ the absolute Galois group of $\mathbb Q$. Here $\overline \rho$ arises from a  Hecke algebra $\mathbb T$ (which is a complete, Noetherian,  local $\mco$-algebra)  acting faithfully on $H^1(X_0(N),\mco)_{\fm}$,  the cohomology of a modular curve $X_0(N)$,  associated to a positive integer $N$,  with coefficients in a discrete valuation ring $\mco$ finite flat over ${\mathbb Z}_p$, localized at the  maximal ideal   $\fm$, and  with $\mathbb T/\fm = k$ a finite field.  There is an  action of $G_{\mathbb Q}$ on $H^1(X_0(N),\mco)_{\fm}$ and  the representation $\overline \rho$  is isomorphic to the representation $\overline \rho_\fm : G_{\mathbb Q} \to GL_2(\mathbb T/\fm)$ associated to $\fm$. This  produces a surjective map  $R \to   \mathbb T$, and the numerical criterion implies in favorable conditions  that $H^1(X_0(N),\mco)_{\fm}$ is free as an $R$-module and that $R$ is a complete intersection. In particular the map $R \to \mathbb T$ is an isomorphism of complete intersections.   In practice,  this is used  to deduce that a certain ring $R'$ (parametrizing deformations of $\overline \rho$ with ramification allowed at a prime $q$) acts freely on $H^1(X_0(Nq^2),\mco)_{\fm'}$, where $\fm'$ is a  maximal ideal of the Hecke algebra acting on $H^1(X_0(Nq^2),\mco)$,  related to $\fm$,  from knowing that  a quotient $R$ of $R'$ acts freely on $H^1(X_0(N),\mco)_{\fm}$.

The main contribution of the present work is a criterion for freeness of a module in terms of its {\em Wiles defect} which was introduced in  \cite{Bockle/Khare/Manning:2021a}; the definition is recalled below. This refines the work of Diamond and Wiles, and also gives a new perspective on these earlier results  from the vantage point of the Wiles defect of rings and of  modules over them. Our proofs differ significantly from those in loc. cit.

The setting for all the results of Wiles and Diamond, and the present paper,  is that there is a commutative, noetherian, local ring $A$ equipped with a surjective map $\lambda\colon A\to \mco$, where $\mco$ is a discrete valuation ring (that will be  fixed throughout the paper), with the property that the conormal module
\[
\con A \colonequals  \fp_A/\fp_A^2\quad \text{where $\fp_A\colonequals \Ker(\lambda)$,}
\]
has finite length as an $\mco$-module. (In the work of Diamond and Wiles, $A$ would be a finite $\mco$-algebra  and $\lambda$ a map of $\mco$-algebras, but we do not impose this.) The \emph{congruence module} of a finitely generated $A$-module $M$ is the $A$-module
\[
\cmod A(M)\colonequals \frac M{M[\fp_A] + M[I_A]}\,,\quad\text{where $I_A\colonequals A[\fp_A]$.}
\]
Here, for any ideal $\fa$ in $A$ we write $M[\fa]$ for $\{m\in M\mid \fa\cdot m=0\}$, the $\fa$-torsion submodule of $M$. As $\fp_A \cdot \cmod A(M)=0$,  the congruence module $\cmod A(M)$ has a natural structure of an $\mco$-module. Moreover the hypothesis that  $\con A$ has finite length implies that the same is true of $\cmod A(M)$. Also $M[\fp_A]$ has a natural structure of an $\mco$-module, and we can consider its rank. Observe that the rank of $M[\fp_A]$ equals the dimension  of $M_{\fp_A}$ over the fraction field of $\mco$, and in particular they are nonzero precisely when $M$ is supported at $\fp_A$. This is the main case of interest in this work.

The \emph{Wiles defect} of the $A$-module $M$ is the integer
\begin{equation}
\label{eq:Wiles-defect}
\delta_{A}(M) =  d\cdot \length{\mco}{\Phi_{A}} -\length{\mco}{\cmod A(M)},
\end{equation}
where $d\colonequals \rank_{\mco} M[\fp_A]$. In \cites{Bockle/Khare/Manning:2021a,Bockle/Khare/Manning:2021b}, this number is divided by $de$, where  $e$ is the ramification index of $\mco$. We  find it more convenient to suppress the denominator. 

We prove:

\begin{theorem}
\label{intro:defect}
Let $M$ be a finitely generated $A$-module with $\depth_AM\ge 1$. There is an equality
\[
\length{\mco}{\cmod A}(M)= (\rank_{\mco} M[\fp_A])\cdot \length{\mco}{\cmod A}(A) - \length{\mco}{(M[\fp_A]/I_AM)}\,.
\]
Equivalently, there is an equality
\[
\delta_{A}(M)=(\rank_{\mco} M[\fp_A])\cdot\delta_A(A)+\length{\mco}{(M[\fp_A]/I_AM)}\,.
\]
In particular $\delta_A(M)\ge 0$. If $M_{\fp_A}\ne 0$ and $\delta_A(M)=0$, then $A$ is complete intersection and $M$ is faithful. When in addition $M$ has rank at most $\rank_{\mco} M[\fp_A]$ at each generic point of $A$, then $M$ is free.
\end{theorem}

The first part of this result, relating the lengths of the congruence modules of $M$ and of $A$, is contained in Theorem~\ref{th:delta-AM}. 
The last part of Theorem~\ref{intro:defect}, describing when $\delta_A(M)=0$ is  suggested  by, and refines, the result of Diamond~\cite[Theorem~2.4]{Diamond:1997}; see Theorem~\ref{th:Diamond} and also the result below. In \cite{Diamond:1997} the module $M$ is required to be finite flat over $\mco$; we replace this by the weaker condition that $M$ is  finitely generated over $A$ and of positive depth.

 One input in its proof is a result of  ~\cite{Fakhruddin/Khare/Ramakrishna:2021} that dealt with the case $M$ is a cyclic $A$-module; see Theorem~\ref{th:Wiles} below.  The main new ingredient is the following criterion for freeness of modules.

\begin{theorem}
\label{intro:Diamond}
Suppose that the ring $A$ is Gorenstein and that $M$ is a finitely generated $A$-module with $\depth_AM\ge 1$.  If 
 \[
 \delta_A(M)= (\rank_{\mco} M[\fp_A]) \delta_A(A)\,,
 \]
and $M$ has rank at most $\rank_{\mco} M[\fp_A]$ at each generic point of $A$, then $M$ is free.
\end{theorem}

This result is implied by  Theorem~\ref{th:prediamond}, where the condition on ranks is replaced by a  weaker one involving  multiplicities.  

The proofs of Theorems~\ref{intro:defect} and \ref{intro:Diamond} are based on a careful study of congruence modules, and various other auxiliary modules related to them. This is the contents of Sections~\ref{se:acat} and \ref{se:congruence}.  In Section~\ref{se:Venkatesh} we give a more streamlined proof of a formula for  $\delta_A(A)$, formulated  (albeit in the setting of certain derived rings) by Venkatesh~\cite{Venkatesh:2006},  and proved in \cite{Appendix:2021}, in terms of certain Andr\'e-Quillen cohomology modules. We end by   explaining how this formula gives another proof of the isomorphism criterion for maps between complete intersection rings due to Wiles~\cite{Wiles:1995} and Lenstra~\cite{Lenstra:1995}.

\begin{chunk} 
We  end the introduction by expanding on the potential significance of Theorem \ref{intro:defect} for the study of congruences between modular forms. The question of comparing congruence modules for $A$ and $M$  has been studied extensively, in the context of  the theory of congruences between modular forms. This theory plays a key role in the  breakthrough work of Wiles~\cite{Wiles:1995}.  As recalled above one studies  a Hecke algebra $\mathbb T$ acting on $H^1(X_0(N),\mco)_{\fm}$ that is isomorphic to $ M \oplus M$ where $M$ is $\mathbb  T$-module that is finite flat over $\mco$,  and hence of positive depth, and of generic rank one as a $\mathbb T$-module. One focuses on an augmentation $\lambda_f=\lambda\colon \mathbb T \to \mco$ arising from a weight 2 newform $f \in S_2(\Gamma_0(N))$. In this context the question of showing that the  congruence modules for $\mathbb T$ and $M$ (associated to the augmentation $\lambda_f$ arising from the newform $f$) are the same has been studied in works of Hida \cite{Hida:1981} and Ribet \cite{Ribet:1983}  in 1980’s and in  many other works, including \cite{Wiles:1995}. The motivation for  doing this   is that the (cohomological) congruence module of $M$ is easier to study and related to a critical value of the  $L$-function associated to the adjoint motive of $f$ (as discovered by Hida in his seminal work), while the congruence module for $\mathbb T$ is more directly related to congruences between $f$ and other newforms in $S_2(\Gamma_0(N))$. In these works, it was shown that the natural surjection   $ \Psi_{\mathbb T}(\mathbb T) \to \Psi_{\mathbb T}(M)  $ is an isomorphism, thus proving that all congruences between $f$ and other newforms in $S_2(\Gamma_0(N))$ are detected ``cohomologically",    by showing that $M$ is a free $\mathbb T$-module (of rank 1).  It follows from our results that such an isomorphism holds precisely when $ M[\fp_A]/I_AM=0$ which can happen without $M$  being free over $\mathbb T$. For instance,   \cite[Theorem 3.12]{Bockle/Khare/Manning:2021b} implies that $ M[\fp_A]/I_AM=0$ when $\End{\mathbb T}M=\mathbb T$, which is a weaker condition than freeness. We hope the observation recorded in the first part of Theorem \ref{intro:defect}, that gives  meaning to the kernel of the  surjective map  $ \Psi_{\mathbb T}(\mathbb T) \to \Psi_{\mathbb T}(M) $, will be useful in the further  study of congruences between modular forms.
\end{chunk}

\begin{ack}
It is a pleasure to thank a referee for a careful reading of the manuscript. The work of the second author was partly supported by the National Science Foundation, under grant DMS-200985.
\end{ack}

\section{The category $\acat$}
\label{se:acat}

\begin{chunk}
\label{ch:catA}
Throughout this work $\mco$ is a discrete valuation ring and $\varpi$  a uniformizing parameter for $\mco$. We write $\acat$ for the category consisting of pairs $(A,\lambda_A)$ where $A$ is a commutative, noetherian, local ring and $\lambda_A\colon A\to \mco$ is a surjective map of rings such that the conormal module $\con A\colonequals \fp_A/(\fp_A)^2$, where $\fp_A\colonequals \Ker(\lambda_A)$, has finite length. The morphisms in $\acat$ are local maps $\vf\colon A\to B$  such that $\lambda_B\vf=\lambda_A$.  The rings in $\acat$ have the same residue field, namely, $\mco/\varpi\mco$.

The condition that the conormal module of $A$ has finite length is equivalent to the natural map $A_{\fp_A}\to \mco_{(0)}$ being an isomorphism; here $\mco_{(0)}$ is the quotient field of $\mco$. Thus  $\fp_A$ is a  minimal prime of $A$ and $\dim A/\fp_A = \dim \mco =1$. This has the following consequence. Subject to the constraint that $\depth A\le 1$,  the pair $(\depth A,\dim A)$ can take all possible values; see Example~\ref{ch:example}.

\begin{lemma}
\label{le:depth}
For any $A\in\acat$ one has $\depth A\le 1\le \dim A$. 
\end{lemma}

\begin{proof}
The claim about dimension is clear from the surjection $A\to \mco$. Any associated prime $\fq$ of $A$ satisfies $\depth A\le \dim A/\fq$; see \cite[Proposition~1.2.13]{Bruns/Herzog:1998}. Setting $\fq\colonequals \fp_A$ yields the upper bound on the depth of $A$.
\end{proof}

Let $A$ be in $\acat$ and  set $I_A\colonequals A[\fp_A]$, the annihilator of $\fp_A$.  In addition the following objects also play an important role in this work:
\[
\cmod A \colonequals \mco/\lambda(I_A) \, \quad\text{and}\quad \frac{I_A}{I_A^2}\,.
\]
The first one is the  \emph{congruence algebra} of $A$, and the last one is the conormal module of the map $A\to A/I_A$.  Since $\fp_A\cdot I_A=0$ the $A$-action on $I_A$ factors through $\mco$, and so the $A$-action on $I_A/I_A^2$ factors through $\cmod A$. By the same token, $\con A$ is a $\cmod A$-module. 

Set $K\colonequals \mco_{(0)}$, the field of fractions of $\mco$.  Viewing $I_A$ as an $\mco$-module, one has 
\begin{equation}
\label{eq:Irank}
(I_A)_{(0)}\cong {\Hom A{\mco}A}_{\fp_A} \cong \Hom{A_{\fp_A}}K{A_{\fp_A}} \cong \Hom KKK\cong K\,.
 \end{equation}
Thus $\rank_{\mco}(I_A)=1$. It also follows that $I_A\not\subseteq \fp_A$, for $\fp_A A_{\fp_A}=0$, so $\lambda(I_A)\ne 0$; equivalently, that $\cmod A$ is a torsion $\mco$-module, that is to say, of finite length. In fact since the Fitting ideal of $\fp_A$ is contained in its annihilator one gets an inequality
\begin{equation}
\label{eq:Wiles-inequality}
\length{\mco}{\con A}\ge \length{\mco}{\cmod A}\,.
\end{equation}
See \cite[Section 5, (5.2.3)]{Darmon/Diamond/Taylor:1997}. Wiles and Lenstra proved that equality holds if and only if the ring $A$ is complete intersection; see Theorem~\ref{th:Wiles}.

Evidently $\fp_A \subseteq A[I_A]$ but in fact equality always holds. 

\begin{lemma}
\label{le:pi-duality}
For any $A$ in $\acat$ one has $\fp_A = A[I_A]$. 
\end{lemma}

\begin{proof}
One has $\lambda(A[I_A]) \cdot \lambda (I_A) = \lambda (A[I_A]\cdot I_A)=0$. Since $\mco$ is a domain and $\lambda(I_A)$ is nonzero, it follows that $\lambda(A[I_A])=0$, that is to say, $A[I_A]\subseteq \fp_A$. 
\end{proof}
\end{chunk}

The following computation will be useful later on. The hypothesis on depth is needed even for the weaker conclusion; see~\ref{ch:depth0}. 

\begin{lemma}
\label{le:depth=1}
Let $A$ be an object in $\acat$ and $M$ a finitely generated $A$-module. When $\depth_AM\ge 1$ one has $ M[\fp_A] \cap M[I_A] =(0)$; in particular $I_AM\cap \fp_AM=0$.
\end{lemma}

\begin{proof}
The $A$-modules $M[\fp_A]$ and $M[I_A]$ are annihilated by $\fp_A$ and $I_A$, respectively. Thus $M[\fp_A] \cap M[I_A]$ is annihilated by $\fp_A + I_A$.
This ideal contains some power of the maximal ideal of $A$, and hence it contains an element that is not a zero divisor on $M$, since the latter has positive depth. Thus $M[\fp_A] \cap M[I_A]=(0)$.
\end{proof}

The conormal module and congruence module are related: Since $A$-acts on $I_A$ through $\mco$ one gets isomorphisms
\begin{equation}
\label{eq:conormal-complement}
\frac{I_A}{I_A^2} \cong I_A\otimes_A \frac A{I_A} \cong I_A \otimes_{\mco} (\mco \otimes_A \frac A{I_A}) \cong I_A\otimes_{\mco} \cmod A\,.
\end{equation}
Here is another expression of the relation between all these invariants
\[
\frac{I_A}{I^2_A}\otimes_{(A/I_A)}\fp_A \cong I_A \otimes_A \fp_A \cong I_A\otimes_{\mco} \frac{\fp_A}{\fp^2_A}\,.
\]
To put this isomorphism in a larger context it helps to remark that 
\[
\fp_A = \Hom A{A/I_A}A \quad\text{and} \quad I_A = \Hom A{\mco}A\,,
\]
where the first equality is by Lemma~\ref{le:pi-duality} and the second is by definition. These are the relative dualizing modules for the maps $A\to A/I_A$ and $A\to \mco$ respectively. Here is one consequence of these observations.

\begin{lemma}
\label{le:conI-free}
Let $A$ be an object in $\acat$. When $\depth A=1$ the ideal $I_A\subset A$ is principal and as a $\cmod A$-module $I_A/I_A^2$ is free of rank one. 
\end{lemma}

\begin{proof}
By Lemma~\ref{le:depth=1} the hypothesis that $\depth A\ge 1$ implies $I_A\cap \fp_A=0$, so the composition $I_A\to A\to \mco$ is injective. Thus as an $\mco$-module $I_A$ is free and of rank one; in particular $I_A\subset A$ is principal. Moreover \eqref{eq:conormal-complement} implies that as a $\cmod A$-module $I_A/I_A^2$ is free of rank one.
\end{proof}

\begin{chunk}
\label{ch:example}
 After suitable completion, any $A$ in $\acat$ is of the form $\mco\pos{\bs x}/J$, for indeterminates $\bs x\colonequals x_1,\dots,x_n$, and  $J\subseteq \varpi (\bs x)+ (\bs x)^2$. One has $\fp_A\colonequals (\bs x)$. Let $f_1,\dots,f_c$ be a minimal generating set for the ideal $J$. For each $i$ there is an unique expression of the form
\[
f_i = \sum a_{ij} x_j + \text{term in $(\bs x)^2$}\,,
\]
with $a_{ij}$ in $\mco$. It is easy to verify that the conormal module of $\lambda_A$ has a presentation
\[
\mco^c \xra{\ (a_{ij}) \ } \mco^n \lra \con A\lra 0\,.
\]
Thus the condition that $\con A$ has finite length is equivalent to $\rank(a_{ij}) = n$; equivalently, $\fitt 0{a_{ij}}$, the zeroth Fitting ideal of the $\mco$-module $\con A$, is nonzero. Using this, or even directly, one can verify that the ring 
\[
A\colonequals \frac{\mco\pos{x_1,\dots, x_n}}{(\varpi x_1,\dots,\varpi x_n)}
\]
is in $\acat$.  The associated primes of $A$ are $(\varpi)$ and $(\bs x)$, so it follows that $\dim A = n$ and  $\depth A=1$.   This shows that for each positive integer $n$, one has rings $A$ in $\acat$  with  $\depth A=1$ and $  \dim A=n$; see  Lemma~\ref{le:depth}. In the same vein the ring $A/x_1(\bs x)$ satisfies $\dim A=n-1$ and $\depth A=0$.
\end{chunk}

The next example shows that the condition  $\depth A\ge 1$ in  Lemma~\ref{le:depth=1} is not superfluous. 

\begin{chunk}
\label{ch:depth0}
When $A\colonequals \mco\pos{x}/(\varpi x,x^2)$ one has 
\[
\fp_A=(x)\subseteq I_A=(\varpi, x)\,.
\]
Thus $I_A\cap \fp_A = \fp_A$. 
\end{chunk}

Next we describe a ring in the category $\acat$ that is Gorenstein, but not complete intersection. This is in anticipation of Theorem~\ref{th:prediamond}.

\begin{chunk}
\label{ch:Gor-not-ci}
Assume $2$ is invertible in $\mco$ and consider the ring
\[
A\colonequals \frac{\mco\pos{x,y,z}} {(x^2-y^2, x^2-z^2, \varpi x - yz, \varpi  y-xz, \varpi  z-xy)}\,.
\]
From \ref{ch:example} one gets that $\con A\cong k^3$, where $k\colonequals \mco/(\varpi)$. In particular $A$ is in $\acat$. We claim that the ring $A$ is reduced, Gorenstein of Krull dimension one, but not complete intersection. It is also finite and free as an $\mco$-module.

Indeed, as an $\mco$-module, $A$  has a basis consisting of (residue classes of) elements $1, x,y,z, x^2$, so that $A$ is finite and free over $\mco$. In particular $\varpi$ is not a zero-divisor on $A$. The ring $A/(\varpi)$ is zero-dimensional, with socle the ideal $(x^2)$, and hence it is Gorenstein; however, it is not a complete intersection, for it has embedding dimension three, but five defining relations; see \cite[Example~3.2.11(b)]{Bruns/Herzog:1998}. Thus $A$ itself is Gorenstein of Krull dimension one, and not complete intersection.

It remains to verify that $A$ is reduced.  A straightforward calculation yields that the prime ideals in $A$ are:
\begin{gather*}
(\varpi, x,y,z)\,, \quad (x,y,z)\,, \quad (x -\varpi, y- \varpi, z- \varpi)\\
(x +\varpi, y+ \varpi, z- \varpi)\,, \quad (x +\varpi, y- \varpi, z+ \varpi)\,, \quad (x -\varpi, y+ \varpi, z+ \varpi)\,.
\end{gather*}
In this list, the first one is the maximal ideal; the rest are minimal. The localization of $A$ at any minimal prime is a field. Since $A$ is Cohen-Macaulay of dimension one, it thus satisfies Serre conditions $(S_1)$ and $(R_0)$, and hence is reduced.

\end{chunk}

\section{Congruence modules}
\label{se:congruence}

In this section we develop basic properties of congruence modules for modules over rings in $\acat$. This prepares us for the next section where we obtain criteria for freeness of the modules in terms of Wiles defects of the modules in question. We begin by recording an observation that will be used multiple times in the sequel.

\begin{lemma}
\label{le:ideal-computation}
Let $J$ be an ideal in a ring $A$ and $M$ an $A$-module. If $x\in A$ is not a zerodivisor on $M$, then 
\[
xM \cap M[J] = x(M[J])\,.
\]
Thus when $A$ is local, $M$ is finitely generated, and $x$  in not a unit in $A$,  either $M[J]=0$ or $M[J]\not \subseteq xM$.
\end{lemma}

\begin{proof}
Indeed, for any $m$ in $M$ if $(xm)\cdot J=0$, then $m \cdot J=0$ for $x$ is not a zerodivisor on $M$, and hence $ m$ is in $M[J]$, as desired. The second part of the claim is by Nakayama's Lemma, applied to $M[J]$.
\end{proof}

\subsection*{Congruence modules}
Let $A$ be an object in $\acat$ and $M$ a finitely generated $A$-module. As in the Introduction, the \emph{congruence module} of $M$ is 
\[
\cmod A(M)\colonequals \frac M{M[\fp_A] + M[I_A]}\,.
\]
We write $\cmod A$, instead of $\cmod A(A)$;  observe that this agrees with the definition introduced in \ref{ch:catA}, for $A[I_A]=\fp_A$, by Lemma~\ref{le:pi-duality}.  Evidently $\cmod A(M)$ is a finitely generated module over the congruence algebra $\cmod A$.  Next we describe a canonical presentation of the congruence module.

Observe that the $A$-modules $M[\fp_A]$ and $M/M[I_A]$ are annihilated by $\fp_A$ and hence are naturally $\mco$-modules. 

\begin{lemma}
\label{le:M-sub-free}
When $\depth_AM\ge 1$, the  $\mco$-modules $M[\fp_A]$ and $M/M[I_A]$ are free of the same rank, and the following natural sequence of $\mco$-modules is exact:
\[
0\lra M[\fp_A] \lra \frac{M}{M[I_A]} \lra \cmod A(M) \lra 0\,.
\]
\end{lemma}

\begin{proof}
The exactness of the sequence is immediate from the definition of $\cmod A$ and Lemma~\ref{le:depth=1}.  The main task is to verify the claims about freeness. Since  $M[\fp_A]$ is an $A$-submodule of $M$ and the latter has positive depth, so does the former. The $A$-action on $M[\fp_A]$ factors through $\mco$ so   $\depth_{\mco}M[\fp_A]\ge 1$ and so $M[\fp_A]$ is free.

Since $\depth_AM\ge 1$ there exists a non-unit element $x\in A$ that is not a zerodivisor on $M$.  We claim that $x$ is not a zerodivisor on $M/M[I_A]$ as well. 

Indeed, suppose $x$ annihilates the residue class in $M/M[I_A]$ of an element $m\in M$, that is to say, $xm\in M[I_A]$. Then 
\[
xm\in xM \cap M[I_A]= x(M[I_A])\,,
\]
where the equality is by Lemma~\ref{le:ideal-computation}. Thus $xm=xm'$ for some $m'\in M[I_A]$, and hence $m=m'$, as $x$ is not a zerodivisor on $M$. Thus $m$ is zero in $M/M[I_A]$.

Since $M/M[I_A]$ has positive depth it too is free as an $\mco$-module.
\end{proof}

\begin{chunk}
\label{ch:bcmod}
We also consider the following variant of the congruence module:
\[
\bcmod A(M)\colonequals \frac M{M[I_A] + I_A M}\,.
\]
Lemma~\ref{le:pi-duality} implies that $\bcmod A(A)=\cmod A$, so that $\bcmod A(M)$ is a $\cmod A$-module, namely a quotient of the $\cmod A$-module $\cmod A\otimes_AM$, and they are equal when  $M$ is free as an $A$-module. It is easy to verify that there is an exact sequence of $\cmod A$-modules:
\begin{equation}
\label{eq:bcmod-cmod}
0\lra \frac{M[\fp_A]}{ I_AM + (M[\fp_A]\cap M[I_A]) } \lra \bcmod A(M) \lra \cmod A(M)\lra 0\,.
\end{equation}
Keep in mind  that $M[\fp_A]\cap M[I_A]=0$ when $\depth_AM\ge 1$, by Lemma~\ref{le:depth=1}.
\end{chunk}

These observations serve to establish the connection between the Wiles defect~\eqref{eq:Wiles-defect} of $A$ and of $M$. This settles  the first part of Theorem~\ref{intro:defect}. As in \emph{op.\ cit.} one could state the equality in terms of congruence modules of $A$ and $M$.

\begin{theorem}
\label{th:delta-AM}
Let $A$ be an object in $\acat$ and $M$ a finitely generated $A$-module. When $\depth_AM\ge 1$, there is an equality
\[
\delta_{A}(M)  = (\rank_{\mco} M[\fp_A])\cdot \delta_A(A) 	+  \length{\mco}{(M[\fp_A]/I_AM)} \,.
\]
In particular $\delta_A(M)\ge 0$.
\end{theorem}

\begin{proof}
Observe that for any $A$-module $M$ there is an isomorphism of $\cmod A$-modules
\[
\bcmod A(M) \cong \cmod A \otimes_{\mco} \frac{M}{M[I_A]}\,.
\]
Since $\depth_AM\ge 1$, the $\mco$-modules $M/M[I_A]$ and $M[\fp_A]$ are free  of the same rank, by Lemma~\ref{le:M-sub-free}, so the isomorphism above yields  the equality
\[
\length{\mco}{\bcmod A(M)} = (\length{\mco}{\cmod A}) d,
\]
where $d\colonequals \rank_{\mco} M[\fp_A]$. This justifies the third equality below.
\begin{align*}
\delta_A(M) 
	&= d\cdot \length{\mco}{\Phi_{A}} -\length{\mco}{\cmod A(M)} \\
	&=d \cdot \length{\mco}{\Phi_{A}} - \length{\mco}{\bcmod A(M)} + \length{\mco}{(M[\fp_A]/I_AM)} \\
	&= d\cdot \length{\mco}{\Phi_{A}} - d \cdot \length{\mco}{\cmod A} + \length{\mco}{(M[\fp_A]/I_AM)} \\
	&= d\cdot \delta_A(A) + \length{\mco}{(M[\fp_A]/I_AM)}\,.
\end{align*}
The first and the last equalities are the definition of defects~\eqref{eq:Wiles-defect} whilst the second one is from \eqref{eq:bcmod-cmod}, and Lemma~\ref{le:depth=1}.  The last equality also uses the observation that the rank of the $\mco$-module $A[\fp_A]=I_A$ is one; see \eqref{eq:Irank}.

The last conclusion holds because $\delta_A(A)\ge 0$; see \eqref{eq:Wiles-inequality}.
\end{proof}

\begin{chunk}
It follows from the proof above that when $\depth_AM\ge 1$ and $\rank_{\mco}M[\fp_A]=1$, there is a surjective map $\cmod A\twoheadrightarrow \cmod A(M)$, with kernel $M[\fp_A]/I_AM$.
\end{chunk}

In the next lemma we show that  the congruence module  for modules $M$  of positive depth remains invariant under pull-back of rings. We use this  in the proof of Theorem \ref{th:Diamond}.

\begin{lemma}
\label{le:delta-invariance}
Let $A\to B$ be a surjective map in $\acat$ and $M$ a finitely generated $B$-module. One has a natural surjection
\[
\cmod B(M) \twoheadrightarrow \cmod A(M)\,.
\]
This map is bijective when $\depth_BM\ge 1$, and then $\delta_A(M)\ge \delta_B(M)$. Moreover $\delta_A(M)=\delta_B(M)$ if and only if $\length{\mco}{\con A}=\length{\mco}{\con B}$.
\end{lemma}

\begin{proof}
Since $\fp_A B=\fp_B$ one has  $I_AB\subseteq I_B$, so there is an exact sequence 
\[
0\lra \frac{I_B}{I_AB} \lra \frac B{I_AB}\lra \frac B{I_B} \lra 0
\]
of $B$-modules. Applying $\Hom B{-}M$ yields an exact sequence
\[
0\lra M[I_B] \lra M[I_A] \lra \Hom B{\frac{I_B}{I_AB}}M
\]
of $B$-modules. Since $M[\fp_B]=M[\fp_A]$, the first part of the statement is immediate for the inclusion $M[I_B]\subseteq M[I_A]$ and the definition of congruence modules. 

Observe that $I_B/I_AB$ is annihilated by $\fp_A$ and $I_A$ and hence it has finite length over $A$, so also over $B$. Thus when $\depth_BM\ge 1$ one has $\Hom B{I_B/I_AB}M=0$, so $M[I_B]=M[I_A]$. This gives the desired isomorphism.

The inequality of Wiles defects is clear since the map $\con A\to\con B$ is surjective~\cite[\S5.2]{Darmon/Diamond/Taylor:1997} as is the statement about equality.
\end{proof}

\section{Criteria for freeness}
In this section we relate the freeness of a module over a local ring in $\acat$ to numerical invariants associated with its congruence module. The main result here is Theorem~\ref{th:prediamond}. In addition to the results on congruence modules presented in Section~\ref{se:congruence}, its proof uses the following criterion for freeness of modules over Gorenstein local rings of Krull dimension zero. 

\begin{lemma}
\label{le:prediamond0}
Let $R$ be a Gorenstein local ring with maximal ideal $\fm$, and of Krull dimension zero. A  finitely generated $R$-module $M$ is free if and only if 
\[
\length RM \le \length R{(R[\fm]\cdot M)} {\length{}R}\,.
\]
\end{lemma}

\begin{proof}
The ideal $R[\fm]$ is the socle of the ring $R$ and since the ring is Gorenstein and of dimension zero, one has $R[\fm]\cong k$; see \cite[Theorem~3.2.10]{Bruns/Herzog:1998}. Both sides of the desired inequality are additive on direct sums of modules, and coincide on $R$, by the remark about socles.
Thus we can assume $M$ has no free summands. Consider the injective hull $M\subseteq F$ of $M$. Since $R$ is Gorenstein of dimension zero, the only indecomposable injective module is $R$ itself; see~\cite[Proposition~3.2.12(e)]{Bruns/Herzog:1998}. Thus the  $R$-module $F$ is free, and since $M$ has not free summands $M\subseteq \fm F$. Therefore
\[
R[\fm] \cdot M \subseteq R[\fm]\cdot (\fm F)=0\,.
\]
The hypothesis thus yields $\length RM=0$ and so $M=0$. 
\end{proof}

In the sequel we use some basic results, recalled below, from the theory of multiplicities for modules over local rings; for details see \cite[Chapter 4]{Bruns/Herzog:1998}.

\begin{chunk}
\label{ch:multiplicity}
Let $A$ be a local ring with maximal ideal $\fm_A$ and $M$ a finitely generated $A$-module. We write $e_A(M)$ for the Hilbert-Samuel multiplicity with respect to $\fm_A$ of the $A$-module $M$; see \cite[\S4.6]{Bruns/Herzog:1998}. For $M=A$ we write $e(A)$ instead of $e_A(A)$. Here are the crucial facts about multiplicities.

When $A$ has Krull dimension zero, then $e_A(M)=\length{A}M$.

When $A$ is a finite $\mco$-algebra, then $e_A(M) = (\rank_{\mco}M)(\rank_{\mco} A)$.

One has $e_A(M)\ge 0$ with strict inequality if and only if $\dim M=\dim A$. This also follows from the additivity formula \eqref{eq:local-multiplicity} below.

Set $\Lambda\colonequals \{\fq\in\mathrm{Spec}\, A\mid \dim (A/\fq)=\dim A\}$; these are the prime ideals corresponding to the components of $\mathrm{Spec}\, A$ of maximal dimension. There is an equality
\begin{equation}
\label{eq:local-multiplicity}
e_A(M) = \sum_{\fq\in \Lambda} (\length{A_\fq}{M_\fq})  e(A_\fq)\,.
\end{equation}
In particular if the rank of $M$ at each $\fq$ in $\Lambda$ is at most that at $\fp_A$, then
\begin{equation}
\label{eq:bound-multiplicity}
e_A(M) \le (\rank_{\mco}{M[\fp_A]})(\sum_{\fq\in \Lambda}  e(A_\fq)) = (\rank_{\mco}{M[\fp_A]})\cdot e(A)\,.
\end{equation}
This holds in particular when $M\colonequals B$ for any surjective map $A\twoheadrightarrow B$ in $\acat$ because one has $B_{\fp_A}=B_{\fp_B}=\mco$, and the rank of $B$ at any $\fq$ in $\Lambda$ is at most one.
\end{chunk}

\begin{lemma}
\label{le:socle}
Let $A\in \acat$ be a Gorenstein ring and $x\in A$ a nonzerodivisor such that $\lambda_A(x)$ is a uniformising parameter for $\mco$. The ring $R\colonequals A/xA$ is Gorenstein with socle equal to $I_A\cdot R$.
\end{lemma}

\begin{proof}
Since $\dim A=1$, by Lemma~\ref{le:depth}, the ring $R$ is zero-dimensional, with maximal ideal $\fm_R \colonequals \fm_AR$.  Rees' theorem~\cite[Lemma~3.1.16]{Bruns/Herzog:1998} yields isomorphisms
\[
\Ext iRkR \cong \Ext {i+1}AkA \quad\text{for all $i$.}
\]
In particular the socle $R[\fm_R]$ of $R$ can be computed as follows:
\[
\Hom RkR\cong \Ext 1AkA \cong \frac{I_A}{\varpi I_A} = I_A\cdot R\,.
\]
The equality holds because $I_A\cap xA = x I_A = \varpi I_A$; see Lemma~\ref{le:ideal-computation}. This justifies the last part  of the result.
\end{proof}

Here is an analogue of Lemma~\ref{le:prediamond0} in the category $\acat$. By \eqref{eq:bound-multiplicity} the upper bound on $e_A(M)$ holds when the rank of $M$ at each generic point of $A$ is at most $d$. Thus the result below contains Theorem~\ref{intro:Diamond} from the Introduction. It was suggested by the arguments in the second half of the  proof of  ~\cite[Theorem~2.4]{Diamond:1997} to prove freeness of modules over complete intersection rings using numerical conditions. There do exist rings in $\acat$ that are Gorenstein but not complete intersection; see  \ref{ch:Gor-not-ci}.

\begin{theorem}
\label{th:prediamond}
Let $A \in\acat$ be a Gorenstein ring, $M$ a finitely generated $A$-module with $\depth_AM\ge 1$, and set $d\colonequals \rank_{\mco}M[\fp_A]$. The $A$-module $M$ is free if, and only if, there are inequalities
\[
\delta_A(M)\le d\cdot \delta_A(A) \qquad\text{and}\qquad e_A(M)\le d\cdot e(A)\,.
\]
\end{theorem}

Given Theorem~\ref{th:delta-AM} the inequality on the left is equivalent to $\delta_A(M)=d\cdot \delta_A(A) $.

\begin{proof}
The ``only if" direct is clear. As to the converse, the hypothesis $\delta_A(M)\le d\cdot \delta_A(A)$ is equivalent to 
$M[\fp_A]=I_AM$, by Theorem~\ref{th:delta-AM}.

Pick an element $x\in A$ mapping to a uniformizer for $\mco$ and such that $x$ is a nonzero divisor on both $A$ and $M$. This is possible as $A$ and $M$ have depth one. Set $R=A/xA$ and $N=M/xM$. Then $R$ is a Gorenstein local ring of dimension zero, with maximal ideal $\fm_A R$ and socle $R[\fm_A R] = I_AR$; see Lemma~\ref{le:socle}. Since $x$ is a nonzero divisor on $A$ and $M$,  and hence on $I_AM$, and not in $\fm_A^2$, one gets
 \[
e(A) = \length{}{R}\,, \quad e_A(M) = \length RN\,  \quad \text{and} \quad e_A(I_AM) = \length R{(I_AM/xI_AM)}\,.
 \]
We also have the following sequence of equalities
 \begin{align*}
 \length R{(I_AR \cdot N)} 
 	&= \length R{\frac{(I_AM + xM)}{xM}} \\
	&= \length R{\frac{I_AM}{(xM\cap I_AM)}} \\
	&=  \length R{(I_AM/x I_AM)}\,,
 \end{align*}
where the third one holds by the observation recorded in Lemma~\ref{le:ideal-computation}, applied with $J=\fp_A$ to $I_{A}M=M[\fp_A]$. Observe that $e_A(I_AM)=e_A(M[\fp_A])=d$, so the hypothesis on multiplicities translates to 
\[
\length RN \le \length R{(I_AR \cdot N)} \length {}R\,.
\]
By Lemma~\ref{le:socle}, the ring $R$ is Gorenstein with socle $I_AR$, so Lemma~\ref{le:prediamond0} applies to yield the $R$-module $N$ is free. Thus the $A$-module $M$ is free.
\end{proof}

Theorem~\ref{th:prediamond} implies  the following  known isomorphism criteria; see \cite[Lemma A.8]{Fakhruddin/Khare/Ramakrishna:2021} and \cite[Theorem 5.21]{Darmon/Diamond/Taylor:1997}. We state it here for ease of reference  as its needed later, and our methods yield a new proof of it.

\begin{corollary}
\label{co:iso-criteria}
Let $\vf\colon A\to B$ be a surjective map in $\acat$. The map $\vf$ is an isomorphism under either of the following conditions:
\begin{enumerate}[\quad\rm(1)]
\item
$\length{\mco}{\cmod A}  = \length{\mco} {\cmod B}$, the ring $A$ is Gorenstein and $B$ is Cohen-Macaulay;
\item
$\length{\mco}{\con A}  = \length{\mco} {\con B}$ and $B$ is complete intersection.
\end{enumerate}
\end{corollary}

\begin{proof}
The goal is to deduce that $B$ is free as an $A$-module, so $\Ker(\vf)=0$.

(1) Since $\cmod B\cong \cmod A(B)$, by Lemma~\ref{le:delta-invariance}, the hypothesis yields $\delta_A(B) = \delta_A(A)$. As $e_A(B)\le e(A)$ always holds---see \ref{ch:multiplicity}---we can apply Theorem~\ref{th:prediamond} to deduce that $B$ is free as an $A$-module. 

(2) Lemma~\ref{le:delta-invariance} yields the first equality below
\[
\delta_A(B) = \delta_B(B) = 0\,.
\]
The second one holds for $B$ is complete intersection. Apply Theorem~\ref{th:prediamond} again.
\end{proof}

\subsection*{Diamond's theorem}
We begin by recalling the result below from \cite[Proposition A.6]{Fakhruddin/Khare/Ramakrishna:2021}; it is used in the proof of Theorem~\ref{th:Diamond}(1).

\begin{theorem}
\label{th:Wiles} 
Let $\vf\colon A \to B$ in $\acat $ be a surjective map  with $\depth B\ge 1$. If $\delta_A(B)= 0$ then  $\vf$  is an isomorphism of complete intersections. \qed
\end{theorem}

Given \eqref{eq:bound-multiplicity}, it is immediate that the result below contains~\cite[Theorem~2.4]{Diamond:1997}, which was proved under the additional hypotheses that $A$ is an $\mco$-algebra and $M$ is a finite free $\mco$-module.

\begin{theorem}
\label{th:Diamond}
Let $A$ be an object in $\acat$ and $M$ a nonzero finitely generated $A$-module supported at $\fp_A$, and with $\depth_AM\ge 1$.
\begin{enumerate}[\quad\rm(1)]
\item
If $\delta_A(M)= 0$, then the ring $A$ is complete intersection, $M$ is faithful,  and $M[\fp_A]=I_AM$.  
\item
If moreover  $e_A(M)\le (\rank_{\mco}M[\fp_A]) e(A)$, then the $A$-module $M$ is free.
\end{enumerate}
\end{theorem}

\begin{proof}
Let $B$ denote the image of $A$ in $\End AM$. Since $M$ is supported at $\fp_A$, the canonical surjection $A\to B$ is in $\acat$. 
Since $M$ has positive depth, so does $\End AM$, and hence also $B$. Thus Lemma~\ref{le:delta-invariance} applies and gives the first inequality below
\[
\delta_A(M)\ge \delta_B(M) \ge 0\,.
\]
The second inequality is by Theorem~\ref{th:delta-AM} applied to $B$. Thus if $\delta_A(M)=0$, then
\[
\delta_A(M)=0=\delta_B(M)\,.
\]
The first equality already implies $M[\fp_A]=I_AM$, by Theorem~\ref{th:delta-AM}. The equality of defects of $M$ over $A$ and $B$ yields that the natural map $\con A\to \con B$ is an isomorphism; see Lemma~\ref{le:delta-invariance}.  This gives the first equality below:
\[
\length{\mco}{\con A} = \length{\mco}{\con B}= \length{\mco}{\cmod B}\,.
\]
Since $\delta_B(M)=0$,  applying Theorem~\ref{th:delta-AM} but now to $M$ viewed as a $B$-module gives $\delta_B(B)=0$, which justifies the second equality. Since $\depth B\ge 1$, the map $A\to B$ is an isomorphism of complete intersections; see Theorem \ref{th:Wiles}. In particular the $A$-module $M$ is faithful.

This completes the proof of the first statement. Given this, the second part is immediate from Theorem~\ref{th:prediamond}.
\end{proof}

\subsection*{An analog of Wiebe's result for modules}
Wiles' theorem characterizing complete intersection rings $A \in \acat$  can be deduced from the theorem of Wiebe~\cite[Theorem~2.3.16]{Bruns/Herzog:1998} that when $R$ is a local ring with maximal ideal $\fm$, its Fitting ideal $\fitt R{\fm}$ is nonzero if and only if $R$ is a complete intersection of dimension zero.  Diamond's theorem suggests the following module theoretic extension of Wiebe's theorem; compare with Lemma~\ref{le:prediamond0}.

\begin{lemma}
\label{le:postdiamond}
Let $R$ be a noetherian local ring, with maximal ideal $\fm$. If $M$ is a nonzero finitely generated $R$-module such that
\[
e(M)\le \length R{(\fitt R{\fm}\cdot M)}\cdot e(R)\,,
\]
then $R$ is a complete intersection with $\dim R=0$, and $M$ is free.
\end{lemma}

\begin{proof}
 As  $ M$ is nonzero, $e(M)\ne 0$ so the hypothesis implies $\fitt R{\fm}$ is nonzero. Thus Wiebe's theorem~\cite[Theorem~2.3.16]{Bruns/Herzog:1998} implies $R$ is complete intersection of Krull dimension zero. Moreover $\fitt R{\fm}=R[\fm]$, the socle of $R$.  At this point we can invoke Lemma~\ref{le:prediamond0} to deduce that $M$ is free.
\end{proof}

\section{Venkatesh's formula}
\label{se:Venkatesh}
In this section we establish a formula for the defect of a local ring in $\acat$, in terms of certain Andr\'e-Quillen homology modules, see Theorem~\ref{th:Venkatesh}. This  gives a  different  proof of  a variant of  the results in ~\cite{Venkatesh:2006} and  ~\cite{Appendix:2021}.

Throughout this section fix $A$ in $\acat$  with $\dim A = 1$, and let 
\[
A\to B \colonequals A/\gam{\fm_A}(A)
\]
be the maximal Cohen-Macaulay quotient of $A$; here $\gam{\fm_A}(A)$ is the $\fm_A$-power torsion submodule of $A$.

\begin{proposition}
\label{pr:gorsum}
Let $\alpha\colon C\to A$ be a surjective map in $\acat$, with $C$ a Gorenstein ring, and set   $I\colonequals \lambda_C(\ann_C(\Ker\alpha))$. With $B$ as above, one has an equality
\[
\length{\mco} {\cmod C} = \length{\mco}{\cmod B} + \length{\mco}{( \mco/ I)}\,.
\]
\end{proposition}

\begin{proof}
First  we reduce to the case $A=B$. Applying $\Hom C-C$ to the exact sequence
\[
0\lra \gam{\fm_A}(A) \lra A\lra B\lra 0\,,
\]
yields an exact sequence
\[
0\lra \Hom CBC \lra \Hom CAC \lra \Hom C{\gam{\fm_A}(A)}C =0
\]
where the equality on the right holds because $\depth C\ge 1$. This gives the equality
\[
 \Hom CBC =  \Hom CAC\,,
\]
so that we can replace $A$ by $B$ and assume  $A$ is Cohen-Macaulay.

Consider the exact sequence
\[
0\lra \Ker \alpha \lra C\lra A\lra 0\,.
\]
Applying $\Hom C{\mco}{-}$ to it yields the exact sequence of $\mco$-modules
\[
0\lra \Hom C{\mco}{\Ker{\alpha}} \lra \Hom C{\mco}C \lra \Hom C{\mco}A \lra \Ext 1C{\mco}{\Ker\alpha}\lra 0
\]
Since the map $C\to \mco$ factors through $A$ one has $\Ker\alpha\subseteq \fp_C$ so
\[
\Hom C{\mco}{\Ker{\alpha}} = I_C \cap \Ker\alpha \ \subseteq \ I_C \cap \fp_C =0\,.
\]
Moreover $\Hom C{\mco}A = \Hom A{\mco}A$ so the exact sequence above becomes
\[
0\lra \eta_C \lra \eta_A \lra \Ext 1C{\mco}{\Ker\alpha}\lra 0\,,
\]
where $\eta_C$ and $\eta_A$ are the images of $I_C$ and $I_A$ in $\mco$. Thus we get an exact sequence
\[
0\lra \Ext 1C{\mco}{\Ker\alpha} \lra \cmod C \lra \cmod A \lra 0\,,
\]
and hence the equality
\begin{equation}
\label{eq:gorsum1}
\length{\mco} {\cmod C} = \length{\mco}{\cmod A} + \length{\mco}{ \Ext 1C{\mco}{\Ker\alpha}}\,.
\end{equation}

Now we analyze the Ext term above. For this it is convenient to work in the stable module category of $C$; see \cite{Buchweitz:2022} for  background. Since $\mco$ is maximal Cohen-Macaulay as a $C$-module, \cite[Corollary 6.4.1]{Buchweitz:2022} gives the first isomorphism below
\[
\Ext 1C{\mco}{\Ker\alpha} \cong \sHom {C}{\mco}{\Omega^{-1} \Ker\alpha} \cong \sHom {C}{\mco}{A} 
\]
The second isomorphism arises from the exact sequence $0\to \Ker \alpha\to C\to A\to 0$. On the other hand, keeping in mind that $A$ is also maximal Cohen-Macaulay as a $C$-module, from Auslander duality~\cite[Theorem~7.7.5]{Buchweitz:2022} one gets that
\[
\length{\mco} {\sHom C{\mco}A} = \length{\mco}{\sHom CA{\mco}}\,.
\]
By the definition of stable homomorphism one has the exact sequence in the top row of the diagram below
\[
\begin{tikzcd}
\Hom CAC\otimes_C \mco \ar[d, equal] \ar[r]  & \Hom CA{\mco} \ar[d, "\cong"] \ar[r] &  \sHom CA{\mco} \ar[d]\ar[r] & 0 \\
\ann_C(\Ker\alpha) \otimes_C \mco \ar[r]  & \mco \ar[r] &  \sHom CA{\mco} \ar[r]  & 0
\end{tikzcd}
\]
Thus $\sHom CA{\mco}\cong \mco/I$. Combining with this \eqref{eq:gorsum1} yields the desired equality.
\end{proof}

\begin{chunk}
\label{ch:ci-approximation}
Let now $\alpha\colon C\to A$ be a surjective map in $\acat$ with $C$ a complete intersection; see \cite[Lemma A.7]{Fakhruddin/Khare/Ramakrishna:2021}. We say that such an $\alpha$ is \emph{minimal} if the natural map $\con C\to \con A$ is bijective; it is always surjective. It is helpful to introduce the ideals
\[
I\colonequals \lambda_C(\ann_C(\Ker\alpha)) \qquad\text{and}\qquad J \colonequals \lambda_C(\fitt C{\Ker\alpha})\,.
\]
In particular, $I\supseteq J$. 
\end{chunk}

Given a map of rings $A\to B$ and a $B$-module $M$ we write $\aqh iBAM$ for the $i$th Andr\'e-Quillen homology module of the $A$-algebra $B$, with coefficients in $M$. We only need these functors for $i=0,1,2$ and Jacobi-Zariski sequence associated to maps; see, for instance,  \cite[\S2]{Brochard/Iyengar/Khare:2020}, or \cite{Iyengar:2007}.

Here is a formula for  $\delta_A(B)$ in terms of these modules; it can also be expressed as an equality of Fitting ideals. See \ref{ch:Venkatesh} for connections with earlier work.

\begin{theorem}
\label{th:Venkatesh}
With notation as above, one has (in)equalities
\[
\delta_A(B) = \length {\mco}{\aqh 2{\mco}{A}{\mco}} - \length{\mco}{(I/J)} \leq \length{\mco}{(\mco/I)}  \,.
\]
Moreover equality holds on the right when $\alpha$ is minimal.
\end{theorem}

\begin{proof}
Since $\aqh 1AC{\mco}= (\Ker{\alpha})\otimes_C \mco$ one gets an equality
\[
\length {\mco}{\aqh 1{A}{C}{\mco}} = \length{\mco}{(\mco/J)}\,,
\]
From this and the Jacobi-Zariski sequence associated to $C\to A\to \mco$, which reads
\[
0 \to \aqh 2{\mco}{A}{\mco} \to  \aqh 1{A}{C}{\mco} \to \con C\lra \con A \to 0\,,
\]
one gets equalities
\begin{align*}
\length {\mco}{\con A} - \length {\mco}{\con C} 
	&=  \length {\mco}{\aqh 2{\mco}{A}{\mco}} -  \length {\mco}{\aqh 1{A}{C}{\mco}} \\
	&= \length {\mco}{\aqh 2{\mco}{A}{\mco}} -  \length{\mco}{(\mco/J)}\,.
\end{align*}
In particular $\length {\mco}{\aqh 2{\mco}{A}{\mco}} -  \length{\mco}{(\mco/J)}\le 0$ with equality when $\alpha$ is minimal; this justifies the inequality and the last assertion in the statement of the theorem. Moreover the equality above yields the second equality below:
\begin{align*}
\length{\mco}{\con A} - \length{\mco}{\cmod B} 
	&=\length{\mco}{\con A} - \length{\mco}{\con C} + \length{\mco}{(\mco/I)} \\
	&= \length {\mco}{\aqh 2{\mco}{A}{\mco}} - \length{\mco}{(\mco/J)} + \length{\mco}{(\mco/I)}  \\
	&= \length {\mco}{\aqh 2{\mco}{A}{\mco}} - \length{\mco}{(I/J)}\,.
\end{align*}
The first one is by Proposition~\ref{pr:gorsum}; it applies as complete intersections rings are Gorenstein, and also $\length{\mco}{\con C}=\length{\mco}{\cmod C}$. The claim about the defect of $B$ as a $A$-module follows.
\end{proof}

\begin{chunk}
\label{ch:Venkatesh}
Suppose that $A$ is a finite $\mco$-algebra, with $\lambda\colon A\to \mco$ a map of $\mco$-algebras. Then it is immediate from the Jacobi-Zariski sequence associated to the map $\mco \to A\to \mco$ that for any integer $i$ one has an isomorphism 
\[
\aqh {i}A{\mco}{\mco}\cong \aqh {i+1}{\mco}{A}{\mco}\,.
\]
Let $K$ be the field of fractions of $\mco$, so that the $\mco$-module $K/\mco$ is the injective hull of the residue field of $\mco$. Then it follows from Matlis duality~\cite[\S3.2]{Bruns/Herzog:1998} that 
\[
\length{\mco}{\aqh {i}A{\mco}{\mco}} \cong \length{\mco}{ \mathrm{D}^{i}(A/\mco;K/\mco)}
\]
where the module on the right is the $i$th Andr\'e-Quillen cohomology of the $\mco$-algebra $A$, with coefficients in $K/\mco$. Thus one can rewrite the equality in Theorem~\ref{th:Venkatesh} as
\[
\length{\mco}{\con A}-\length{\mco}{\cmod A(B)} = \length {\mco}{ \mathrm{D}^{1}(A/\mco;K/\mco)} - \length{\mco}{(I/J)} 
\]
It is in this form that the formula was proposed by Venkatesh~\cite{Venkatesh:2006}, and proved in ~\cite{Appendix:2021}. From our perspective the avatar in terms of Andr\'e-Quillen homology is more natural.
\end{chunk}

Theorem~\ref{th:Venkatesh} expresses the defect of $A$ as a difference of two positive integers. It is not clear why they are both zero when the defect is zero, as asserted by Wiles' theorem~\ref{th:Wiles}. What is more Venkatesh's formula only applies when $\dim A=1$. So we sketch an argument that deduces the latter result from the former, though only under the additional hypothesis that $\dim B=1$.  

\begin{proof}[Proof of Theorem~\ref{th:Wiles} when $\dim B=1$.]
We start by reducing to the case where $\dim A=1$. Set $\fb=\Ker(A\to B)$ and $A'\colonequals A/\fb^2$. Thus the map $\vf$ factors through the surjection $A\to A'$. This gives the second of the following inequalities:
\[
\length{\mco}{\cmod { B'}} \ge \length{\mco}{\con A} \ge \length{\mco}{\con {A'}}\,.
\]
The first one is by hypothesis. Thus the hypothesis of the desired result applies to the surjection $A'\to B$; we claim it suffices to verify the conclusion for this map, for if this map is an isomorphism one gets that $\fb=\fb^2$, so that $\fb=0$, that is to say $\vf$ is an isomorphism, as desired. Since $\dim A'=\dim B$ we can replace $A$ by $A'$ and assume $\dim A=1$.

Next we reduce to the case where $B$ is $A$ modulo its $\fm_A$-power torsion ideal, so we can apply Venkatesh's equality: Set $B' \colonequals A/\gam{\fm_A}(A)$.  Since $\vf$ is surjective, $\vf(\fm_A)=\fm_B$ so that $\vf(\gam{\fm_A}(A))$ is $\fm_B$-power torsion; thus $\depth B\ge 1$ implies $\vf(\gam{\fm_A}(A))=0$, that is to say, $\vf$ factors through the surjection $A\to B'$. This gives the first inequality below:
\[
\length{\mco}{\cmod { B'}} \ge \length{\mco}{\cmod B} \ge \length{\mco}{\con A}\,.
\]
The second one is part of the hypothesis.  Since $\depth B'\ge 1$ as well, the map $A\to B'$ also satisfies the hypothesis of the desired result. We claim that it suffices to verify then that $A$ is complete intersection. Indeed then $A= B'$ and keeping in mind that the lengths of $\cmod A$ and $\con A$ coincide, we get the inequality:
\[
\length{\mco}{\cmod B}  \ge \length{\mco}{\cmod A}\,.
\]
Then Corollary~\ref{co:iso-criteria} applies to yield that $\vf$ is an isomorphism. Thus we can assume $B= B'$, which puts us in the context of Theorem~\ref{th:Venkatesh}.

Since $\dim A=1$ we can choose a minimal presentation $\alpha\colon C\to A$, with $C$ a complete intersection; see~\ref{ch:ci-approximation}.   Theorem~\ref{th:Venkatesh} with $M=B$ yields
\[
\length{\mco}{\con A} - \length{\mco}{\cmod B} = \length{\mco}{(\mco/I)}\,.
\]
By the hypothesis, the term on the left is negative so we deduce that $I=\mco$. Therefore $\Ker\alpha=0$, so  $A=C$ and $A$ is complete intersection. 
\end{proof}

\end{document}